\documentclass[a4paper,12pt]{amsart}   
\usepackage[english]{babel}
\usepackage{amsthm,amsmath,amssymb,amscd}
\usepackage{shuffle}
\usepackage{xcolor}
\allowdisplaybreaks

\usepackage{setspace}
\usepackage{mathtools}
\mathtoolsset{showonlyrefs}

\newtheorem{theorem}{Theorem}[section]
\newtheorem{proposition}[theorem]{Proposition}
\newtheorem{definition}[theorem]{Definition}
\newtheorem{lemma}[theorem]{Lemma}
\newtheorem{corollary}[theorem]{Corollary}

\newtheorem{problem}[theorem]{Problem}

\theoremstyle{remark}
\newtheorem{remark}[theorem]{Remark}

\numberwithin{equation}{section}

\DeclareFontEncoding{OT2}{}{}
\DeclareFontSubstitution{OT2}{cmr}{m}{n}
\DeclareFontFamily{OT2}{cmr}{\hyphenchar\font45 }
\DeclareFontShape{OT2}{cmr}{m}{n}{<5><6><7><8><9>gen*wncyr<10><10.95><12><14.4><17.28><20.74><24.88>wncyr10}{}
\DeclareFontShape{OT2}{cmr}{b}{n}{<5><6><7><8><9>gen*wncyb<10><10.95><12><14.4><17.28><20.74><24.88>wncyb10}{}
\DeclareMathAlphabet{\mathcyr}{OT2}{cmr}{m}{n}
\DeclareMathAlphabet{\mathcyb}{OT2}{cmr}{b}{n}
\SetMathAlphabet{\mathcyr}{bold}{OT2}{cmr}{b}{n}

\input cyracc.def

\newcommand{\Q}{\mathbb Q}
\newcommand{\Z}{\mathbb Z}
\newcommand{\x}{\mathsf{x}}
\newcommand{\cQ}{\mathcal{Q}}
\newcommand{\cP}{\mathcal{P}}
\newcommand{\cL}{\mathcal{L}}
\newcommand{\cG}{\mathcal G}
\newcommand{\cV}{\mathcal V}
\newcommand{\dm}{\mathfrak{dm}}
\newcommand{\ls}{\mathfrak{ls}}

\DeclareMathOperator{\Ad}{Ad}
\DeclareMathOperator{\ad}{ad}
\DeclareMathOperator{\gr}{gr}
\DeclareMathOperator{\cp}{\underline{\circ}}
\DeclareMathOperator{\Lie}{Lie}
\DeclareMathOperator{\sh}{\mathcyr {sh}}
\DeclareMathOperator{\opast}{\ast}

\title[Ecalle's and Brown's solutions to double shuffle equations]{On Ecalle's and Brown's polar solutions to the double shuffle equations modulo products}
\author{Nils Matthes}
\address{Department of Mathematics, Kyushu University, 744 Motooka, Nishi-ku, Fukuoka, 819-0395, Japan}

\address{Current address (N.M.): Mathematical Institute, University of Oxford, Andrew Wiles Building, Radcliffe Observatory Quarter, Woodstock Road, Oxford, OX2 6GG, United Kingdom \newline 
	E-mail address: {\tt nils.matthes@maths.ox.ac.uk}}

\author{Koji Tasaka}
\address{Department of Information Science and Technology, Aichi Prefectural University, 1522-3 Ibaragabasama, Nagakute-shi, Aichi, 480-1198, Japan \newline
	E-mail address: {\tt tasaka@ist.aichi-pu.ac.jp}}


\date{\today}

\begin{document}

\begin{abstract}
Two explicit sets of solutions to the double shuffle equations modulo products were introduced by Ecalle and Brown respectively. We place the two solutions into the same algebraic framework and compare them. We find that they agree up to and including depth four but differ in depth five by an explicit solution to the linearized double shuffle equations with an exotic pole structure.
\end{abstract}
\maketitle

\section{Introduction}

In this paper, we compare two explicit sets of solutions to the double shuffle equations modulo products introduced by Brown \cite{Brown:Anatomy} and Ecalle \cite{Ecalle:ARI}, respectively. 
One of our main purposes is to put the two constructions on an equal footing, thereby highlighting similarities (and differences) of the two approaches. 
We first explain the relevance of the problem under consideration in the context of multiple zeta values and then state our main results.

\subsection{Motivation}
Multiple zeta values are defined for positive integers $n_1,\ldots,n_r$ with $n_r\geq 2$ by the nested sum
\begin{equation} \label{eqn:mzv}
\zeta(n_1,\ldots,n_r):=\sum_{0<m_1<\cdots<m_r}\frac{1}{m_1^{n_1}\cdots m_r^{n_r}}.
\end{equation}
We call $n_1+\dots+n_r$ the weight and $r$ the depth, and $1\in \Q$ is regarded as the multiple zeta value of weight 0 and depth 0. 
One of the ultimate goals of the study of multiple zeta values is to find a presentation of the $\Q$-algebra $\mathcal{Z}$ generated by all multiple zeta values, or equivalently, to describe all $\Q$-algebraic relations among the numbers \eqref{eqn:mzv}. This problem is related to explicitly describing the action of the Tannakian fundamental group of the category of mixed Tate motives over $\Z$ on the unipotent fundamental group of $\mathbb P^1 \setminus \{0,1,\infty\}$ \cite{Brown:Anatomy,DeligneGoncharov:MTM}.
For more details, as well as a relation with the Grothendieck--Teichm\"uller group \cite{Drinfeld} and the Kashiwara--Vergne group \cite{AlT}, we refer to \cite{Furusho:around}.

Now recall the dimension conjecture on $\mathcal{Z}$.
Let $\mathcal{Z}_k \subset \mathcal{Z}$ be the $\Q$-vector subspace spanned by all multiple zeta values of weight $k$. 
We have $\mathcal{Z}_0=\Q$, $\mathcal{Z}_1=\{0\}$ and it is easily seen that $\dim_{\Q}\mathcal{Z}_k \leq 2^{k-2}$, for $k\geq 2$. 
In \cite{Zagier}, Zagier proposed the conjecture stating that $\dim_{\Q}\mathcal{Z}_k\stackrel{?}{=}d_k$, where the $d_k$ are defined by the generating series $\sum_{k=0}^{\infty}d_kt^k=\frac{1}{1-t^2-t^3}$. 
Goncharov \cite{Goncharov1} and Terasoma \cite{Terasoma} proved independently that $\dim_{\Q}\mathcal{Z}_k \leq d_k$, i.e. the numbers $d_k$ give at least an upper bound for $\dim_{\Q}\mathcal{Z}_k$. 
Since $d_k\approx 0.4115\ldots \times (1.3247\ldots)^k$, this shows that there are in fact numerous $\Q$-linear relations among the numbers \eqref{eqn:mzv}.
Several classes of such relations are known, such as motivic relations \cite{Brown:MTM}, associator relations \cite{Furusho:dsh} and (regularized) double shuffle relations \cite{IKZ} which, by results of Drinfeld, Goncharov and Furusho, are related as follows (cf. \cite{Furusho:around})
\begin{equation} \label{eqn:rels}
\{\mbox{regularized double shuffle relations}\} \subset \{\mbox{associator relations}\} \subset \{\mbox{motivic relations}\}.
\end{equation}
Moreover, it is known that the inequality $\dim_{\Q}\mathcal{Z}_k \leq d_k$ can be attained using motivic relations only \cite{Brown:MTM}. However, the definition of the motivic relations is rather technical. Giving another proof of the inequality $\dim_\Q \mathcal{Z}_k \le d_k$ using a different class of relations is not only a very challenging project, but also particularly important in connection with motive theory since it might yield a more elementary characterization of the motivic relations.

\subsection{Problem}
In this paper, we focus on the regularized double shuffle relations which can be described as certain functional equations for polynomials in the ring $\Q\langle \x_0,\x_1\rangle$ of non-commutative polynomials in two variables $\x_0,\x_1$. 
Denote by $\mathfrak{dmr}_0$ the subspace of $\Q\langle \x_0,\x_1\rangle$ consisting of solutions to the regularized double shuffle equations modulo products and $\zeta(2)$ (see \cite[\S\S3.3.1-2]{Racinet} or \cite[Appendix A]{Furusho:dsh} for the definition). 
It is graded by the weight, where $\x_0,\x_1$ both have weight $-1$. 
The space $\mathfrak{dmr}_0$ is a graded Lie algebra under the Ihara bracket (cf. \cite[Proposition 4.A.i]{Racinet}).
In this paper, following \cite{Brown:Anatomy,Ecalle:ARI,IKZ}, we will embed $\mathfrak{dmr}_0$ into the space $\cP:=\bigoplus_{r\geq 0}\Q[x_1,\ldots,x_r]$
of finite sequences of polynomials and always consider $\mathfrak{dmr}_0$ as a subspace of $\cP$ under this embedding. For example, the element $[\x_0,[\x_0,\x_1]]+[\x_1,[\x_1,\x_0]] \in \mathfrak{dmr}_0$ corresponds to the sequence $(0,x_1^2,-2x_1+x_2,0,\ldots)$.

Since the numbers \eqref{eqn:mzv} satisfy the regularized double shuffle relations there is a non-canonical surjection of $\Q$-algebras 
\[
U(\mathfrak{dmr}_0)^{\vee} \otimes_{\Q} \Q[\zeta(2)] \rightarrow \mathcal{Z},
\]
which is believed to be an isomorphism (cf. \cite[Conjecture 1]{IKZ}).
Here, we denote by $U(\mathfrak{dmr}_0)^{\vee}$ the graded dual of the universal enveloping algebra of $\mathfrak{dmr}_0$. 
We ask for a presentation of the Lie algebra $\mathfrak{dmr}_0 $ in terms of explicit generators and relations. By a result of Goncharov \cite{Goncharov2}, we know that $\mathfrak{g}^{\mathfrak{m}} \subset \mathfrak{dmr}_0$ where $\mathfrak{g}^{\mathfrak{m}}$ denotes the (image in $\cP$ of the) motivic Lie algebra (see e.g. \cite[Definition 2.3]{Brown:depthgraded} for the definition) and Brown \cite{Brown:MTM} proved that $\mathfrak{g}^{\mathfrak{m}}$ is freely generated by non-canonical elements $\sigma_{2k+1} $ of weight $-2k-1\ (k\ge1)$. 
It is expected that the equality $\mathfrak{g}^{\mathfrak{m}}\stackrel{?}{=}\mathfrak{dmr}_0$ holds. A positive solution to this would show that the inequality $\dim_{\Q}\mathcal{Z}_k \leq d_k$ can be attained using the regularized double shuffle relations only. 
Another consequence would be that all inclusions in \eqref{eqn:rels} are actually equalities which would yield in particular a more elementary characterization of the motivic relations.
Therefore, one might expect to obtain explicit formulas for a choice of the generators $\sigma_{2k+1} $ by solving the regularized double shuffle equations.

So far, no explicit formula for $\sigma_{2k+1}$ is known except for the canonical elements $\sigma_3,\sigma_5,\sigma_7,\sigma_9$; the first example is $\sigma_3:= (0,x_1^2,-2x_1+x_2,0,\ldots)$. Furthermore, it is also known that $\sigma_{2k+1}=(0,x_1^{2k},\ldots)$ for all $k\geq 3$ \cite{Brown:depthgraded,Ihara:Tatetwists}. The problem we are attacking in this paper is as follows (see also \cite[Problem 1]{Brown:depth3}).
\begin{problem} \label{prob:precise}
Find explicit formulas for (some choice of) the $\sigma_{2k+1}$ by solving the (regularized) double shuffle equations.
\end{problem}

\subsection{Polar solutions of Ecalle and Brown and our main results}

In \cite{Brown:Anatomy}, Brown introduced a family of explicit solutions to the double shuffle equations modulo products \eqref{eqn:dm} in all odd weights $\leq -3$, and applied it to the study of $\sigma_{2k+1}$. A similar construction was made by Ecalle \cite{Ecalle:ARI}, and our main result concerns a comparison of the two approaches.

The solutions constructed by Brown and Ecalle are elements of a certain graded vector space
\begin{equation}
\cQ \subset \prod_{r\geq 0}\Q(x_1,\ldots,x_r)
\end{equation}
which contains the polynomial subspace $\cP$ (see \S2 for the precise definition of $\cQ$). In fact the solutions constructed by Ecalle and Brown are not contained in $\cP$; they will be thus called polar solutions. Nevertheless, removing poles, one may well be able to obtain an expression of $\sigma_{2k+1}$ as a kind of ``Taylor expansion" (a suitable linear combination) in terms of their polar solution, so that one can tackle Problem \ref{prob:precise}.
This expression is called ``anatomical" decompositions of zeta elements by Brown \cite[\S11]{Brown:Anatomy}, where a list of anatomical decompositions of $\sigma_{2k+1}$ for $k\le 4$ is provided (see also \S6).

To construct their polar solutions, Ecalle and Brown actually worked on slightly different spaces, so to compare them we need to place their constructions into the same space. For this, the set $\dm_\cQ$ of solutions to the double shuffle equations modulo products in $\cQ$, is introduced in \S2.5.
This enlarges Brown's space $\mathfrak{pdmr}$ defined in \cite[Definition 9.1]{Brown:Anatomy} (see Remark \ref{rmk:dm} for the difference). 
Note that the space $\dm_\cQ$ does not take into account the regularization, corresponding to the ``correction" factor $\psi_{\rm corr}$ in \cite[\S3.3.1]{Racinet}, and hence $\mathfrak{dmr}_0$ is not a subspace of $\dm_\cQ$.
We also consider the homogenized versions of the double shuffle equations modulo products, which are known as the linearized double shuffle equations \cite[\S8]{IKZ}, and denote by $\ls_\cQ \subset \cQ$ the subset of solutions to the linearized double shuffle equations (see Definition \ref{dfn:ls}). The following theorem is basically due to Racinet, Ecalle and Brown:

\begin{theorem}\label{thm:dmls}
The spaces $\dm_\cQ$ and $\ls_\cQ$ are Lie algebras under the Ihara bracket $\{\phantom{\cdot} ,\phantom{\cdot}\}$.
\end{theorem}

The Ihara bracket will be recalled in \S4.1.
For the proof of Theorem \ref{thm:dmls}, there is another shorter exposition given by Brown (see Remark \ref{rmk:another-proof}), but we repeat quickly Ecalle's theory of moulds \cite{Ecalle:ARI} from \cite{SalernoSchneps,Schneps:ARI}, since it already involves Ecalle's construction of polar solutions.
On this side, another Lie bracket, called the ari bracket $\{\phantom{\cdot} ,\phantom{\cdot}\}_{\rm ari}$ (see \S3.2), comes into play.
We will first show an explicit connection with the Ihara bracket in Proposition \ref{prop:V} and then rephrase Theorem 7.2 of \cite{SalernoSchneps} in our notation.
As a consequence, we obtain the explicit Lie isomorphism (see \S5.1)
\[\chi_E: \ls_\cQ \stackrel{\cong}\longrightarrow \dm_\cQ. \]

Since $\ls_\cQ$ is bigraded by the weight and the depth, it is easier to compute elements in $\ls_\cQ$ for fixed weights and depths. 
For example, we have $x_1^{2k}=(0,x_1^{2k},0,0,\ldots)\in \ls_\cQ$ for all $k\in \Z$.
The map $\chi_E$ is then used to lift these elements to solutions to the double shuffle equations modulo products, thereby constructing explicit elements of $\dm_\cQ$ in all depths.
A crucial feature is that the element $\chi_E(f)\in \dm_\cQ$ in general will have poles, even though $f\in \ls_\cQ$ is a sequence of polynomials.

A similar definition was made by Brown \cite[\S14.2]{Brown:Anatomy}; he defines an injective $\Q$-linear map (see Definition \ref{dfn:constructionChiE})
\[
\chi_B: \ls_\cQ \longrightarrow \cQ,
\]
and announces that $\chi_B(\ls_\cQ) \subset \dm_\cQ$.\footnote{This would actually imply $\chi_B(\ls_{\cQ})=\dm_\cQ$, see the discussion after Theorem \ref{thm:Brown} below.}
Furthermore, the map $\chi_B$ is expected to be a Lie isomorphism, which is still open. In any case, we can apply $\chi_B$ to the depth one elements $x_1^{2k}$ to obtain a second set of explicit solutions to the double shuffle equations modulo products.

We can now state our main result.

\begin{theorem} \label{thm:main}
For any $f  \in \ls_\cQ$, we have
\[ \chi_E(f) \equiv \chi_B(f) \mod \cQ^{(d+4)},\]
where $d$ is minimal such that $f^{(d)}\neq 0$ and $\cQ^{(r)}$ is defined in \eqref{eqn:Q-depth}.
\end{theorem}

Theorem \ref{thm:main} says that the images $\chi_E(f)$ and $\chi_B(f)$ agree up to depth $d+3$. The proof of Theorem \ref{thm:main} is straightforward; one can write down the formulas for $\chi_E(f)$ and $\chi_B(f)$ explicitly, and then compare them. In particular, for $f=x_1^{2k}$ we obtain $\chi_E(x_1^{2k})^{(d)}=\chi_B(x_1^{2k})^{(d)}$ for $d\leq 4$. On the other hand, $\chi_E(x_1^{2k})^{(5)}-\chi_B(x_1^{2k})^{(5)} \neq 0$ and the difference can be written using a specific solution to the linearized double shuffle equations in depth four (see Remark \ref{rmk:diff}).

One can also wonder about the existence of a ``universal'' isomorphism between $\ls_\cQ$ and $\dm_\cQ$ from which both $\chi_E$ and $\chi_B$ can be obtained as specializations.
We hope to address this problem in a future work (see \S 5.4).

\subsection{Structure of the paper}

In \S2, we introduce the double shuffle equations in the setting of rational functions, thereby also fixing some of our notation. 
In \S3, we review Ecalle's theory of moulds.
In \S4, we prove Proposition \ref{prop:V}: an explicit connection between the Ihara and the ari bracket. 
These are then put to use in \S5 where, after reviewing the definition of the maps $\chi_E$ and $\chi_B$, we prove our main result, Theorem \ref{thm:main}.
\S6 will be devoted to a list of anatomical decompositions of $\sigma_{2k+1}$ of Ecalle's polar solutions.

\subsection*{Acknowledgments}
We are grateful to Francis Brown for inspiring and very valuable discussions, as well as to Leila Schneps for answering our questions about Ecalle's theory. We would like to thank Henrik Bachmann, Francis Brown, Hidekazu Furusho, Ulf K\"uhn as well as the referee for comments and corrections on earlier versions, and the Hausdorff Institute for Mathematics and the Max Planck Institute for Mathematics for hospitality. This work is partially supported by JSPS KAKENHI Grant No. 18K13393, 17F17020 and 16H07115 and was done while N.M. was a JSPS postdoctoral fellow. N.M. would also like to thank his academic host, Professor Masanobu Kaneko, for his support and fruitful discussions on this research.

\section{The double shuffle equations for rational functions} \label{sec:2}
In this section, we define two spaces $\dm_\cQ$ and $\ls_\cQ$, partially following \cite{Brown:Anatomy}.

\subsection{Moulds} \label{ssec:2.1}
We consider sequences $f=(f^{(0)},f^{(1)}(x_1),f^{(2)}(x_1,x_2),\ldots)$ of rational functions where $f^{(r)} \in \Q(x_1,\ldots,x_r)$ (in particular, $f^{(0)} \in \Q$ is a constant). In \cite{SalernoSchneps}, these objects are called ``rational-function valued moulds'' and they are a special case of the moulds introduced by Ecalle \cite{Ecalle:ARI}; a mould is a sequence of multivariable functions. For any such sequence, we will call $f^{(r)}$ the depth $r$ component of $f$. We shall actually only be interested in a graded subspace of all sequences of rational functions, and accordingly define
\begin{align*}
\cQ:=\bigoplus_{k\in \Z}\cQ_k  \quad \mbox{with} \quad  \cQ_k = \left\{ (f^{(r)})\in  \prod_{r\ge0} \Q(x_1,\ldots,x_r) \ \bigg| \ \deg f^{(r)} =k-r, \ r\ge0\right\},
\end{align*}
where $\Q(x_1,\ldots,x_r)$ is regarded as $\Q$ when $r=0$.
We shall assume that the degree of $0\in\Q$ can be arbitrary.
We say that an element $f\in \cQ_k$ has weight $-k$. 
The space $\cQ$ is naturally a topological $\Q$-vector space whose topology is induced from the \textit{depth filtration} $\cQ^{\bullet}$, defined for $d\geq 1$ by 
\begin{equation}\label{eqn:Q-depth}
\cQ^{(d)}:=\{ (f^{(r)})\in \cQ \, \vert \, f^{(r)}=0, \, 0\le r\le d-1 \}
\end{equation}
and $\cQ^{(0)}=\cQ$.
Write $\cQ^{(d)}_k = \cQ^{(d)} \cap \cQ_k$.
Then we can and will identify the quotient space $\cQ^{(d)}_k/\cQ^{(d+1)}_k$ with the subspace $\{f\in \Q(x_1,\ldots,x_d) \mid \deg f=k-d\}\subset \Q(x_1,\ldots,x_d)$.

For the rest of this paper, by a ``mould'' we will always mean an element of $\cQ$.

\subsection{Convention} \label{ssec:2.1.1}
For any mould $f=(f^{(r)})$ and any non-empty word $\x_{n_1}\cdots \x_{n_r}$ consisting of any letters indexed by positive integers, we write
\[ f(\x_{n_1}\cdots \x_{n_r})=f^{(r)}(x_{n_1},\ldots,x_{n_r}) ,\]
and set $f(\varnothing)=f^{(0)}$ for the empty word $\varnothing$.
We extend this notation to all linear combinations of words by linearity with the coefficients in the rational functions.
For example, we have
\begin{equation}\label{eq:ex-ast-dep2}
f(\x_1\x_2+\x_2\x_1+\frac{\x_1-\x_2}{x_1-x_2}) = f^{(2)}(x_1,x_2)+f^{(2)}(x_2,x_1)+\frac{f^{(1)}(x_1)-f^{(1)}(x_2)}{x_1-x_2}.
\end{equation}
This notation is very useful for our purpose.
Note that there are exceptions; for example, if $f^{(2)}(x_1,x_2)$ has a pole at $x_1=x_2$, then $f(\x_1\x_1)$ does not make sense. 
However, we do not meet this sort of substitution later.

\subsection{Two products} \label{ssec:2.2}

Let $X=\{\x_1,\x_2,\ldots\}$ be an infinite set of letters. We also denote by $K:=\varinjlim\limits_r \Q(x_1,\ldots,x_r)$, where the transition maps are the obvious inclusions, the field of rational functions in a countable number of variables.
Denote by $\langle X\rangle$ the free monoid of words on $X$, including the empty word $\varnothing$, and for any $\Q$-algebra $R$ we denote by $R\langle X\rangle$ the free $R$-module spanned by the set $\langle X\rangle$.

The shuffle product $\sh$ is defined on $\Q\langle X\rangle$ inductively by 
\[ \x_n w_1\sh  \x_m w_2 = \x_n (w_1 \sh  \x_m w_2) + \x_m (\x_n w_1\sh  w_2), \]
and $\varnothing\sh w=w\sh  \varnothing=w$, for any words $w,w_1,w_2 \in \langle X\rangle$ and $n,m\ge1$, and then extending by $\Q$-bilinearity.
For example, we have $f(\x_1 \sh \x_2) = f^{(2)}(x_1,x_2)+f^{(2)}(x_2,x_1)$.

We also define the product $\ast$ on $K\langle X\rangle$ inductively by
\[ \x_{n}w_1 \ast \x_{m}w_2= \begin{cases} 0 & n=m\\  \x_{n}(w_1\ast \x_{m}w_2)+\x_{m}(\x_{n}w_1\ast w_2)+\frac{\x_{n}-\x_{m}}{x_{n}-x_{m}}(w_1\ast w_2)& \mbox{otherwise} \end{cases}  \]
with $\varnothing\ast w=w \ast \varnothing= w$, for any words $w_1,w_2 \in \langle X\rangle$ and $n,m\ge1$, and then extending $K$-bilinearly.
In the expansion of $\x_{n}w_1 \ast \x_{m}w_2$, we shall think of $x_n$ and $x_m$ as commutative variables corresponding to $\x_n$ and $\x_m$, respectively.
For example, the right-hand side of \eqref{eq:ex-ast-dep2} can be written as a single term $f(\x_1\ast \x_2)$. Note that the product $\ast$ corresponds to the stuffle product on the level of commutative generating series (see \cite[\S7]{Ihara:RIMS} and \cite[\S5]{Brown:Anatomy}).

\subsection{The sharp operator} \label{ssec:2.2.1}

We define an injective $\Q$-linear map $\sharp: \cQ \rightarrow \cQ,\ f\mapsto f^\sharp$ for each depth $r$ component by
\begin{equation} \label{eqn:sharp}
(f^{\sharp})^{(r)}(x_1,\ldots,x_r):=f^{(r)}(x_1,x_1+x_2,\ldots,x_1+\cdots+x_r),
\end{equation}
which appears for example in \cite[p.331]{IKZ} (where the order of variables is reversed).
Let $\flat: \cQ \rightarrow \cQ, \ f\mapsto f^\flat$ be the injective $\Q$-linear map, given by
\begin{equation} \label{eqn:flat}
(f^{\flat})^{(r)}(x_1,\ldots,x_r):=f^{(r)}(x_1,x_2-x_1,\ldots,x_r-x_{r-1}).
\end{equation}
We understand $f^\sharp(\varnothing)=f^\flat(\varnothing)=f^{(0)}$.
Since $(f^\sharp)^\flat=(f^\flat)^\sharp =f$, the maps $\sharp$ and $\flat$ are automorphisms on $\cQ$.

\subsection{Double shuffle equations modulo products} \label{ssec:2.2.2}
We can now define our fundamental object $\dm_\cQ$ studied in this paper.
\begin{definition} \label{dfn:dm}
Let $\dm_{\cQ}$ be the set of elements $f\in \cQ$ such that (i) $f^{(0)}=0$, (ii) its depth one component is even, i.e. $f^{(1)}(-x_1)=f^{(1)}(x_1)$, and (iii) for all $r\geq 2$ and $1\leq i<r$, we have
\begin{equation} \label{eqn:dm}
f^{\sharp}(\x_1\cdots \x_i\sh \x_{i+1}\cdots \x_r)=f(\x_1\cdots \x_i\opast \x_{i+1}\cdots \x_r)=0.
\end{equation}
\end{definition}

The equations \eqref{eqn:dm} are called the double shuffle equation modulo products.
By definition, the space $\dm_\cQ$ inherits the weight grading and the depth filtration from the ambient space $\cQ$.
Note that the Lie algebra $\mathfrak{dmr}_0$ introduced by Racinet \cite[\S3.3.1]{Racinet} is \textit{not} embedded into our space $\dm_\cQ$. 
The difference is that, while the equation $f(w_1\opast w_2)=0$ holds on the nose for elements in $\dm_\cQ$, elements in $\mathfrak{dmr}_0$ are only required to satisfy the above equation up to a certain correction term.

\begin{remark}\label{rmk:dm}
In \cite[Definition 9.1]{Brown:Anatomy}, a space $\mathfrak{pdmr}$ very similar to $\dm_\cQ$ is introduced.
It differs in a restriction to poles; the subspace $\mathfrak{pdmr}\subset \dm_\cQ$ admits only poles along $x_i=0$ and $x_i=x_j$ for $i,j\ge1$.
In \cite[\S10]{Brown:Anatomy}, certain solutions $\psi_{2n+1}\in \mathfrak{pdmr}$ were constructed explicitly, which will be different from Ecalle's solution $\chi_E(x_1^{2n})$ and Brown's solution $\chi_B(x_1^{2n})$ defined in \S5. 
We point out that there seems to be no restriction on the depth one component in the definition of $\mathfrak{pdmr}$, although this is required in the proof of its Lie algebra structure \cite[Theorem 9.2]{Brown:Anatomy}.

\end{remark}

\subsection{Linearized double shuffle equations} \label{ssec:2.2.3}
Observe that $f(\x_1\dots\x_i\ast \x_{i+1}\dots\x_r) \equiv f(\x_1\dots\x_i\sh \x_{i+1}\dots\x_r)$ modulo terms of depth $\leq r-1$. The equations \eqref{eqn:ls}, defined below, which are obtained from \eqref{eqn:dm} by discarding all terms of depth $\leq r-1$ are called the linearized double shuffle equations \cite{Brown:depthgraded,IKZ}.
\begin{definition} \label{dfn:ls}
Let $\ls_{\cQ}$ be a set of elements $f\in \cQ$ such that (i) $f^{(0)}=0$, (ii) its depth one component is even, i.e. $f^{(1)}(-x_1)=f^{(1)}(x_1)$, and (iii) for all $r\geq 2$ and $1\leq i<r$, we have
\begin{equation} \label{eqn:ls}
f^{\sharp}(\x_1\cdots \x_i\sh\x_{i+1}\cdots \x_r)=f(\x_1\cdots \x_i\sh\x_{i+1}\cdots \x_r)=0.
\end{equation}
\end{definition}
A prototype of solutions to \eqref{eqn:ls} is given by the generating series of regularized multiple zeta values modulo products and lower depths of a fixed weight $k$ (see \cite[Corollary 7]{IKZ}).

Note that, by definition, the subspace $\ls_{\cQ} \subset \cQ$ is bigraded by the weight and the depth, namely, we have
\[ \ls_\cQ \cong \prod_{r\ge0} \gr^{(r)} \ls_\cQ,\]
where $\gr^{(r)} \ls_\cQ:= \ls_\cQ^{(r)}/\ls_\cQ^{(r+1)}$ with $\ls_\cQ^{(r)}:=\ls_\cQ \cap \cQ^{(r)}$.
Later, for $f\in \ls_\cQ$, its depth $r$ component $f^{(r)}$ is viewed as an element in $\gr^{(r)}\ls_\cQ$.


\section{Some aspects of Ecalle's theory} \label{sec:3}
In this section, we briefly review Ecalle's theory from \cite{Ecalle:ARI,SalernoSchneps,Schneps:ARI} with our notational conventions and translate it into our spaces.

\subsection{$ARI$ and $GARI$}
The space $\cQ$ contains two distinguished subsets
\[
\cL:=\{ f\in \cQ \mid f^{(0)}=0 \}, \quad \cG:=\{ f\in \cQ \mid f^{(0)}=1 \}
\]
of sequences whose depth zero term $f^{(0)}$ is equal to $0$, respectively equal to $1$ (note that $\cL=\cQ^{(1)}$ is a graded $\Q$-vector space with respect to the weight).
The notations $\cL$ and $\cG$ come from ``Lie algebra" and ``group", corresponding to restrictions of $ARI$ and $GARI$ to rational functions, respectively (see \cite[\S9]{Ecalle:ARI} and also \cite[\S2]{SalernoSchneps} for the definition of $ARI$ and $GARI$).
Below, we will endow $\cL$ with a Lie bracket which turns it into a (pro-nilpotent) Lie algebra, and $\cG$ will be the associated (pro-unipotent) group (more precisely, its $\Q$-rational points). 

\subsection{The ari bracket}
We now define the ari bracket \cite{Ecalle:ARI,SalernoSchneps} translated into our setting.
The order of composition is reversed from \cite[Eq.(2.10)]{SalernoSchneps}, for easier comparison with the Ihara bracket.
\begin{definition}\label{def:ari-bracket}
The preari action is the continuous, $\Q$-bilinear map $\cp_{\rm ari}: \cL \times \cL \rightarrow \cL $.
The depth $d$ component of $f\cp_{\rm ari} g$ is given by 
\[ (f\cp_{\rm ari} g)^{(d)} (x_1,\ldots,x_{r+s})= \sum_{r+s=d} (f^{(r)}\cp_{\rm ari} g^{(s)})(x_1,\ldots,x_{r+s}),\] 
where the terms $f^{(r)}\cp_{\rm ari} g^{(s)}$ are defined by
\begin{equation*}\label{eq:def-preari_u}
\begin{aligned}
&(f^{(r)}\cp_{\rm ari}g^{(s)})(x_1,\ldots,x_{r+s})\\
&= \sum_{i=0}^s f^{(r)}(x_{i+1},\ldots,x_{i+r}) g^{(s)}(x_1,\ldots,x_i,\sum_{j=i+1}^{i+r+1}x_j, x_{i+r+2},\ldots, x_{r+s})  \\
&-\sum_{i=1}^s f^{(r)}(x_{i+1},\ldots,x_{i+r}) g^{(s)}(x_1,\ldots,x_{i-1},\sum_{j=i}^{i+r}x_j, x_{i+r+1},\ldots, x_{r+s}).
\end{aligned}
\end{equation*}
The ari bracket $\{\phantom{\cdot} ,\phantom{\cdot}\}_{\rm ari}$ is then defined to be the antisymmetrization of the preari action, i.e.
\[
\{f,g\}_{\rm ari}:=f\cp_{\rm ari} g-g\cp_{\rm ari}f.
\]
\end{definition}

The space $\cL$ forms a graded Lie algebra with $\{\phantom{\cdot},\phantom{\cdot}\}_{\rm ari}$ (see \cite[Proposition 2.2.2]{Schneps:ARI}). 
Rephrasing Theorem 3.3 of \cite{SalernoSchneps} using our conventions, we now prove that the image $\ls_\cQ^{\sharp}:=\{ f^{\sharp} \in \cL\mid f\in \ls_\cQ \}$ forms a Lie subalgebra of $\cL$.
\begin{theorem}\label{thm:lie-ls}
The space $\ls_\cQ^{\sharp}$ is a Lie algebra under the ari bracket $\{\phantom{\cdot} ,\phantom{\cdot}\}_{\rm ari}$.
\end{theorem}
\begin{proof}
This follows from the fact that $ARI_{\underline{al}/\underline{al}}$ forms a Lie algebra under the ari bracket $\{\phantom{\cdot} ,\phantom{\cdot}\}_{\rm ari}$ \cite[Theorem 3.3]{SalernoSchneps} and the equivalence of the defining equations of $ \ls_\cQ^{\sharp}$ and $ARI_{\underline{al}/\underline{al}}$, which we now check.

Let us first recall $ARI_{\underline{al}/\underline{al}}$ from \cite[\S3]{SalernoSchneps}, with our convention.
Let $swap^{\flat}:\cQ\rightarrow \cQ$ be the $\Q$-linear map given for each depth $r$ component by
\begin{align*} 
swap^{\flat}(f)^{(r)}(x_1,\ldots,x_r) = f^{(r)} (x_r,x_{r-1}-x_r,\ldots,x_1-x_2).
\end{align*}
The space $ARI_{\underline{al}/\underline{al}}$ is then defined as the set of moulds $f$ such that $f^{(0)}=0,f^{(1)}(x_1)=f^{(1)}(-x_1)$ and for all $r\ge2$ and $1\le i<r$ the equality $f(\x_1\cdots \x_i\sh \x_{i+1}\cdots \x_r)=swap^\flat (f)(\x_1\cdots \x_i\sh \x_{i+1}\cdots \x_r)=0$ holds.
Now let $anti: \cQ\rightarrow \cQ$ be the $\Q$-linear map given by 
\begin{equation}\label{eqn:anti}
anti(f)^{(r)}(x_1,\ldots,x_r)=f^{(r)}(x_r,\ldots,x_1).
\end{equation}
We easily see that $swap^\flat(f) =anti(f^\flat)$ and that $anti(f)(w)=f(\overline{w})$ for any word $w$, where $\overline{w}$ denotes the reverse word of $w$.
We extend the notation $\overline{w}$ to all linear combination of words by linearity.
Using the well-known identity $w\x_n\sh w'\x_m=(w\sh w\x_m)\x_n + (w\x_n\sh w')\x_m$, we see by induction on the length of words that
\[ \overline{w\sh w'} = \overline{w} \sh \overline{w}'\]
holds for any words $w,w'$.
Hence
\[swap^\flat (f)(\x_1\cdots \x_i\sh \x_{i+1}\cdots \x_r) = f^\flat (\x_i\cdots \x_1\sh \x_r\cdots \x_{i+1}), \]
which shows that the defining equations of $ARI_{\underline{al}/\underline{al}}$ and $\ls_\cQ^\sharp$ are the same.
\end{proof}

\subsection{The ari-exponential} \label{ssec:3.3}
For $f \in \cL$ and $n\geq 1$, define $f_{\cp_{\rm ari}}^n$ recursively by
\[
f_{\cp_{\rm ari}}^1=f, \quad f_{\cp_{\rm ari}}^n=f\cp_{\rm ari} f_{\cp_{\rm ari}}^{n-1}.
\]
Again, we have reversed the order of composition with respect to \cite[\S5]{SalernoSchneps}.
\begin{definition}\label{dfn:expari}
The ari-exponential $\exp_{\rm ari}: \cL\rightarrow \cG$ is defined for $f\in \cL$ by
\[
\exp_{\rm ari}(f)=\mathbf{1}+\sum_{n=1}^{\infty}\frac{f_{\cp_{\rm ari}}^n}{n!},
\]
where $\mathbf{1}:=(1,0,0,\ldots)\in \cG$.
\end{definition}
The ari-exponential is bijective, and its inverse is denoted by 
\[
\log_{\rm ari}: \cG \rightarrow \cL.
\]
Note that $\log_{\rm ari}(F)$ for $F\in \cG$ is computed inductively by  $\exp_{\rm ari}(\log_{\rm ari}(F))=F$. For instance, one computes $\log_{\rm ari}(F)^{(0)}=0, \log_{\rm ari}(F)^{(1)}=F^{(1)}, \log_{\rm ari}(F)^{(2)}= F^{(2)} - \frac12 F^{(1)} \cp_{\rm ari} F^{(1)}$ and so on.
\begin{remark}
Due to the non-associativity of $\cp_{\rm ari}$, we have to be slightly careful about the definition of the exponential and logarithm map. In particular, $\log_{\rm ari}(F)$ is \textit{not} given by the standard formula $\sum_{n=1}^{\infty}\frac{(-1)^n}{n}(F-\mathbf{1})_{\cp_{\rm ari}}^n$. The reason is that the product $\circledast$ on the universal enveloping algebra $\mathcal{U}(\cL)$ of $\cL$ is not equal to $\cp_{\rm ari}$ (cf. \cite[Definition 2.9]{OudomGuin}). On the other hand, the definition of $\exp_{\rm ari}$ is fine, because if we identify $\cL$ as a subspace of $\mathcal{U}(\cL)$ (via Poincar\'e--Birkhoff--Witt), then $f\circledast g=f\cp_{\rm ari}g$, whenever $f\in \cL$.
\end{remark}

\subsection{Adjoint action} \label{ssec:3.4}
Using $\exp_{\rm ari}$, we can endow $\cG$ with a group structure using the Baker--Campbell--Hausdorff formula (cf. \cite[Ch. 3]{Reutenauer}).
More precisely, let ${\rm ch}(f,g)$ denote the Baker--Campbell--Hausdorff series of $f$ and $g$ (see also \cite[\S5]{SalernoSchneps}).
One can define the group law $\circ_{\rm ari}$ on $\cG$ for $F=\exp_{\rm ari}(f)$ and $G=\exp_{\rm ari}(g)$ with $f,g \in \cL$ by
\[
F \circ_{\rm ari} G:=\exp_{\rm ari}({\rm ch}(f,g)).
\]
The group $(\cG,\circ_{\rm ari})$ then acts on the Lie algebra $(\cL,\{\phantom{\cdot},\phantom{\cdot}\}_{\rm ari})$ via the adjoint action below.
\begin{definition}
The adjoint action $\Ad_{\rm ari}: \cG \times \cL \rightarrow \cL$ is defined for $G\in\cG$ and $f\in\cL$ by the formula
\[
\Ad_{\rm ari}(G)(f):= \sum_{n=0}^{\infty}\frac{1}{n!}\ad_{\rm ari}^n(g)(f),
\]
where we let $g=\log_{\rm ari}(G)$ and $\ad_{\rm ari}^0(g)(f)=f$ and set $\ad_{\rm ari}^n(g)(f)= \{\ad_{\rm ari}^{n-1}(g)(f),g\}_{\rm ari} $ for $n\ge1$.\footnote{ Note that with our convention $\ad(g)$ is a right action; it is related to the usual adjoint map $\widetilde{\ad}(g)(f):=\{g,f\}$ via $\ad(g)(f)=\widetilde{\ad}(-g)(f)$.}
\end{definition}

It follows that for every fixed $G \in \cG$ we obtain an isomorphism of Lie algebras $\Ad_{\rm ari}(G): \cL \rightarrow \cL$.
In Ecalle's theory, a suitable element $P\in \cG$ is chosen and its adjoint action induces a Lie isomorphism between $\ls_\cQ^\sharp$ and $\dm_\cQ^\sharp=\{f^\sharp \in \cL \mid f\in \dm_\cQ\}$, so that, by Theorem \ref{thm:lie-ls} the space $\dm_\cQ^\sharp$ inherits a Lie algebra structure under the ari bracket.

\subsection{Ecalle's theorem} \label{ssec:4.2}

We follow the exposition of \cite[\S4]{Schneps:ARI} for the choice of $P$.
Let us define $P\in \cG$ recursively by $P^{(0)}=1$ and for $r\geq 1$ by
\[
P^{(r)}(x_1,\ldots,x_r)=\frac{1}{x_1+\cdots+x_r}\sum_{i=0}^{r-1} P^{(i)}(x_1,\ldots,x_i)d^{(r-i)}(x_{i+1},\ldots,x_r)
\]
where $d^{(r)} \in \Q(x_1,\ldots,x_r)$ is defined by $d^{(0)}=0$ and for $r\geq 1$ by
\[
d^{(r)}(x_1,\ldots,x_r)=\frac{B_r}{r!}\sum_{i=0}^{r-1}(-1)^i\binom{r-1}{i}\frac{x_{r-i}}{x_1\cdots x_r},
\]
and the $B_r$ are the Bernoulli numbers. 
In particular, $d^{(1)}(x_1)=-\frac 12$ and $d^{(r)}$ vanishes for odd $r\geq 3$. 
Clearly, $P$ has weight 0.
The element $P$ is denoted by \textit{pal} in both \cite{Ecalle:ARI} and \cite{Schneps:ARI}, and $d$ is denoted by \textit{dupal}. 

For simplicity of notation, we let
\[
\phi_0=\log_{\rm ari}(P) \in \cL.
\]
The first few values of $\phi_0=(\phi_0^{(r)})$ are given by
\begin{equation}\label{eqn:phi_0}
\begin{aligned}
\phi_0^{(1)}(x_1)&=-\frac{1}{2x_1}, \ \phi_0^{(2)}(x_1,x_2)=\frac{x_1-x_2}{12x_1x_2(x_1+x_2)},\\
\phi_0^{(3)}(x_1,x_2,x_3)&=\frac{-x_2 x_1^2+x_3 x_1^2-x_2^2 x_1+x_3^2 x_1+2 x_2 x_3 x_1-x_2
   x_3^2-x_2^2 x_3}{48 x_1 x_2x_3  (x_1+x_2) (x_2+x_3) (x_1+x_2+x_3)}.
\end{aligned}
\end{equation}

\begin{remark}\label{rmk:phi0}
Combining Proposition 4.12, Theorems 4.2.1 and 4.3.4 of \cite{Schneps:ARI}, for any $r\ge0$ and $0\le i\le r$ we have
\begin{equation} \label{eqn:dshmodlowerdepth}
\begin{gathered}
P(\x_1\cdots \x_i\sh  \x_{i+1}\cdots \x_{r})=P(\x_1\cdots \x_i)P(\x_{i+1}\cdots \x_{r}),\\
P^{\flat}(\x_1\cdots \x_i\sh  \x_{i+1}\cdots \x_{r})=P^\flat(\x_1\cdots \x_i)P^\flat(\x_{i+1}\cdots \x_{r}).
\end{gathered}
\end{equation}
From the first equality in \eqref{eqn:dshmodlowerdepth}, one can show that 
\begin{equation}\label{eq:phi0-shuffle}
\phi_0 (\x_1\cdots \x_i \sh \x_{i+1}\cdots \x_r) = 0
\end{equation}
holds for all $r\ge2$ and $1\leq i<r$ (see \cite[Proposition 2.6.1]{Schneps:ARI}).
On the other hand, note that $\phi_0^\flat$ does not satisfy the equations $\phi_0^\flat (\x_1\cdots \x_i \bullet \x_{i+1}\cdots \x_r) = 0$ for $\bullet\in \{\ast,\sh\}$, and hence, $\phi_0^\flat\not\in\dm_\cQ$ and $\phi_0^\flat\not\in\ls_\cQ$.
\end{remark}

With $P$, one of the main theorems of Ecalle's theory of moulds is stated as follows. 
\begin{theorem}\label{thm:Ecalle's-fundamental}
For $f\in \ls_{\cQ}^{\sharp}$, we have $\Ad_{\rm ari}(P)(f) \in \dm_{\cQ}^{\sharp}$. 
In particular, $\Ad_{\rm ari}(P): \cL \rightarrow \cL$ restricts to an isomorphism of Lie algebras
\[
\Ad_{\rm ari}(P):\ls_\cQ^{\sharp} \stackrel{\cong}\longrightarrow \dm_\cQ^{\sharp},
\]
where both sides are endowed with the ari bracket.
\end{theorem}
\begin{proof}
This follows from \cite[Theorem 7.2]{SalernoSchneps} (see also \cite[Theorem 4.6.1]{Schneps:ARI}, where they actually show that $ARI_{\underline{al}/\underline{al}}\cong ARI_{\underline{al}/\underline{il}}$ in the first place).
We should check that the defining equations of $ARI_{\underline{al}/\underline{il}}$ and $\dm_\cQ^{\sharp}$ are the same, but this follows from the same argument as in the proof of Theorem \ref{thm:lie-ls}.
\end{proof}

As a corollary, we see that 
\begin{corollary}\label{cor:lie-dm}
The space $\dm_{\cQ}^{\sharp}$ is a Lie algebra under the ari bracket.
\end{corollary}

\section{Proof of Theorem \ref{thm:dmls}}

The goal of this section is to give a proof that $\ls_\cQ$, $\dm_\cQ$ are Lie algebras with the Ihara bracket.

\subsection{The Ihara bracket}
Following \cite{Brown:Anatomy}, we define the Ihara bracket.

\begin{definition}\label{def:Ihara-bracket}
The linearized Ihara action is the continuous, $\Q$-bilinear map $\cp: \cL \times \cL \rightarrow \cL$.
The depth $d$ component of $f\cp g$ is given by 
\[ (f\cp g)^{(d)} (x_1,\ldots,x_{r+s})= \sum_{r+s=d} (f^{(r)}\cp g^{(s)})(x_1,\ldots,x_{r+s}),\] 
where the terms $f^{(r)}\cp g^{(s)}$ are defined by
\begin{equation} \label{eqn:linIhara}
\begin{aligned}
&(f^{(r)}\cp g^{(s)}) (x_1,\ldots,x_{r+s})\\
&=\sum_{i=0}^s f^{(r)} (x_{i+1}-x_i,\ldots,x_{i+r}-x_i)g^{(s)} (x_1,\ldots,x_i,x_{i+r+1},\ldots,x_{r+s}) \\
& +(-1)^{r}\sum_{i=1}^s f^{(r)}(x_{i+r}-x_{i+r-1},\ldots,x_{i+r}-x_i)g^{(s)}(x_1,\ldots,x_{i-1},x_{i+r},\ldots,x_{r+s}).
\end{aligned}
\end{equation}
The Ihara bracket $\{\phantom{\cdot} ,\phantom{\cdot}\}$ is then defined to be the antisymmetrization of the linearized Ihara action, i.e.
\[
\{f,g\}:=f\cp g-g\cp f.
\]
\end{definition} 
Note that if $f\in \cL_k$, i.e. $f$ has weight $-k$ then the right hand side of \eqref{eqn:linIhara} is equal to
\begin{equation*}
\begin{aligned}
&\sum_{i=0}^s f^{(r)} (x_{i+1}-x_i,\ldots,x_{i+r}-x_i)g^{(s)} (x_1,\ldots,x_i,x_{i+r+1},\ldots,x_{r+s}) \\
& +(-1)^{\deg f^{(r)}+r}\sum_{i=1}^s f^{(r)}(x_{i+r-1}-x_{i+r},\ldots,x_i-x_{i+r})g^{(s)}(x_1,\ldots,x_{i-1},x_{i+r},\ldots,x_{r+s}),
\end{aligned}
\end{equation*}
which is precisely the formula for the linearized Ihara action given in \cite[\S6.3]{Brown:Anatomy}.

The Ihara bracket $\{\phantom{\cdot} ,\phantom{\cdot}\}$ provides another Lie structure on $\cL$ (cf. \cite[Lemma 6.6]{Brown:Anatomy}), and is equivalent to a Lie bracket first considered by Ihara \cite{Ihara:Tatetwists} in the polynomial case.

\begin{remark}
A useful way to memorize the definition of the linearized Ihara action is that the terms appearing in the definition can be identified with the set of all possible ways of excising from $(x_1,\ldots,x_{r+s})$ a sub-tuple $(x_{i+1},\ldots,x_{i+r})$ of $r$ successive points. Then one either subtracts from the removed tuple its left neighbor $x_i$, or one subtracts from the right neighbor $x_{i+r+1}$ the tuple (in which case there is an additional weighting factor $(-1)^r$). 
\end{remark}

\subsection{Comparison of Lie brackets}
We show an explicit connection between the Ihara bracket $\{\phantom{\cdot},\phantom{\cdot}\}$ and the ari bracket $\{\phantom{\cdot},\phantom{\cdot}\}_{\rm ari}$. 
We begin with the following well-known lemma.

\begin{lemma}\label{lem:anti}
If a mould $f$ satisfies $f(\x_1\cdots \x_i\sh\x_{i+1}\cdots \x_r)=0$ for all $1\le i<r$, we have
\[ f(\x_1\x_2\cdots \x_r)+(-1)^{r}f(\x_r\cdots \x_2\x_1)=0.\]
\end{lemma}
\begin{proof}
This is a standard fact about the shuffle algebra (see e.g. \cite[Lemma 2.5.3]{Schneps:ARI}).
\end{proof}

Define $\cV \subset \cL$ to be the $\Q$-vector subspace of elements $f$ which satisfy
$
f+\varphi(f)=0,
$
where $\varphi$ is the linear involution of $\cL$, given for each depth $r$ component by
\[
\varphi(f)^{(r)}(x_1,\ldots,x_r)=(-1)^rf^{(r)}(x_r,\ldots,x_1).
\]
Note that by Lemma \ref{lem:anti}, we see that both $\dm_\cQ^\sharp$ and $\ls_\cQ^\sharp$ are subspaces of $\cV$.
We first prove that the space $\cV$ is a Lie subalgebra of $\cL$ with the ari bracket.
\begin{proposition} \label{prop:closedunderari}
The $\Q$-vector space $\cV$ is closed under $\{\phantom{\cdot},\phantom{\cdot}\}_{\rm ari}$.
\end{proposition}
\begin{proof}
We need to show that, given $f,g \in \cL$ such that
$
f+\varphi(f)=g+\varphi(g)=0,
$
we have
\begin{equation} \label{eqn:Vstable}
\{f,g\}_{\rm ari}+\varphi(\{f,g\}_{\rm ari})=0.
\end{equation}
It is enough to check this for every component $f^{(r)}$ of $f$ and $g^{(s)}$ of $g$ separately, so, for simplicity of notation, let $f=f^{(r)}$ and $g=g^{(s)}$. 
Unraveling the definition of $\{\phantom{\cdot},\phantom{\cdot}\}_{\rm ari}$, we see that $\{f,g\}_{\rm ari}+\varphi(\{f,g \}_{\rm ari})$ equals
\begin{equation}\label{eqn:1}
\begin{aligned}
\Bigg( &\sum_{i=0}^sf(x_{i+1},\ldots,x_{i+r})g(x_1,\ldots,x_i,\sum_{j=i+1}^{i+r+1}x_j,x_{i+r+2},\ldots,x_{r+s})\\
&-\sum_{i=1}^s f(x_{i+1},\ldots,x_{i+r})g(x_1,\ldots,x_{i-1},\sum_{j=i}^{i+r}x_j,x_{i+r+1},\ldots,x_{r+s})\\
&-\sum_{i=0}^rg(x_{i+1},\ldots,x_{i+s})f(x_1,\ldots,x_i,\sum_{j=i+1}^{i+s+1}x_j,x_{i+s+2},\ldots,x_{r+s})\\
&+\sum_{i=1}^r g(x_{i+1},\ldots,x_{i+s})f(x_1,\ldots,x_{i-1},\sum_{j=i}^{i+r}x_j,x_{i+r+1},\ldots,x_{r+s})\Bigg)\\
&+(-1)^{r+s}\Bigg(\sum_{i=0}^s f(x_{r+s-i},\ldots,x_{s-i+1})g(x_{r+s},\ldots,x_{r+s-i+1},\sum_{j=s-i}^{r+s-i}x_j,x_{s-i-1},\ldots,x_1)\\
&-\sum_{i=1}^sf(x_{r+s-i},\ldots,x_{s-i+1})g(x_{r+s},\ldots,x_{r+s-i+2},\sum_{j=s-i+1}^{r+s-i+1} x_j,x_{s-i},\ldots,x_1)\\
&-\sum_{i=0}^r g(x_{r+s-i},\ldots,x_{r-i+1})f(x_{r+s},\ldots,x_{r+s-i+1},\sum_{j=r-i}^{r+s-i} x_j,x_{r-i-1},\ldots,x_1)\\
&+\sum_{i=1}^rg(x_{r+s-i},\ldots,x_{r-i+1})f(x_{r+s},\ldots,x_{r+s-i+2},\sum_{j=r-i+1}^{r+s-i+1}x_j,x_{r-i},\ldots,x_1)\Bigg).
\end{aligned}
\end{equation}
Using that $f+\varphi(f)=g+\varphi(g)=0$, we see that the first and the sixth, the second and fifth, the third and eighth as well as the fourth and the seventh sum simplify, and therefore \eqref{eqn:1} equals
\begin{equation*}
\begin{aligned}
&f(x_{s+1},\ldots,x_{r+s})g(x_1,\ldots,x_s)+(-1)^{r+s}f(x_r,\ldots,x_1)g(x_{r+s},\ldots,x_{r+1})\\
&-g(x_{r+1},\ldots,x_{r+s})f(x_1,\ldots,x_r)-(-1)^{r+s}g(x_s,\ldots,x_1)f(x_{r+s},\ldots,x_{s+1})=0,
\end{aligned}
\end{equation*}
where we again used that $f+\varphi(f)=g+\varphi(g)=0$, and \eqref{eqn:Vstable} follows.
\end{proof}

\begin{proposition}\label{prop:V}
For all $f,g\in \cV$, we have
$
\{f,g\}_{\rm ari}^{\flat}=\{ f^{\flat},g^{\flat} \}.
$
In particular, the space $\cV^{\flat}:=\{ f^{\flat} \in \cL\mid f\in \cV \}$ is closed under the Ihara bracket $\{\phantom{\cdot},\phantom{\cdot}\}$ and the map $f\mapsto f^{\flat}$ induces an isomorphism of Lie algebras 
\begin{equation}\label{eqn:isom}
(\cV,\{\phantom{\cdot},\phantom{\cdot}\}_{\rm ari}) \cong (\cV^{\flat}, \{\phantom{\cdot},\phantom{\cdot}\} ).
\end{equation}
\end{proposition}
 \begin{proof}
We will prove the slightly stronger result
\begin{equation} \label{eqn:stronger}
(f\cp_{\rm ari}g)^{\flat}=f^{\flat}\cp g^{\flat}, \quad \mbox{for all } f\in \cV, \, g\in \cL.
\end{equation}
Computing the left hand side of \eqref{eqn:stronger} gives
\begin{equation}
\begin{aligned}
&\sum_{i=0}^s f^{(r)}(x_{i+1:i},\ldots,x_{i+r:i+r-1}) g^{(s)}(x_1,\ldots,x_{i:i-1},x_{i+r+1:i}, x_{i+r+2:i+r+1},\ldots, x_{r+s:r+s-1})  \\
&-\sum_{i=1}^s f^{(r)}(x_{i+1:i},\ldots,x_{i+r:i+r-1}) g^{(s)}(x_1,\ldots,x_{i-1:i-2},x_{i+r:i-1}, x_{i+r+1:i+r},\ldots, x_{r+s:r+s-1})
\end{aligned}
\end{equation}
where $x_{j:k}:=x_j-x_k$. On the other hand, the right hand side of \eqref{eqn:stronger} is equal to
\begin{equation}
\begin{aligned}
&\sum_{i=0}^s f^{(r)} (x_{i+1:i},x_{i+2:i+1},\ldots,x_{i+r:i+r-1})g^{(s)} (x_1,\ldots,x_{i:i-1},x_{i+r+1:i},\ldots,x_{r+s:r+s-1}) \\
& +(-1)^{r}\sum_{i=1}^s f^{(r)}(x_{i+r:i+r-1},x_{i+r-1:i+r-2},\ldots,x_{i+1:i})g^{(s)}(x_1,\ldots,x_{i-1:i-2},x_{i+r:i-1},\ldots,x_{r+s:r+s-1}).
\end{aligned}
\end{equation}
Finally, using that $f+\varphi(f)=0$, we get
\begin{equation}
-f^{(r)}(x_{i+1:i},x_{i+2;i+1},\ldots,x_{i+r:i+r-1})=(-1)^{r}f^{(r)}(x_{i+r:i+r-1},\ldots,x_{i+2;i+1},x_{i+1:i})
\end{equation}
and \eqref{eqn:stronger} follows. The fact that $\cV^{\flat}$ is a Lie algebra under the Ihara bracket follows from this together with Proposition \ref{prop:closedunderari}.
\end{proof}

We note that the isomorphism \eqref{eqn:isom} does not extend to an isomorphism between $(\cL,\{\phantom{\cdot},\phantom{\cdot}\})$ and $(\cL,\{\phantom{\cdot},\phantom{\cdot}\}_{\rm ari})$, e.g. for $f(x_1,x_2)=x_1+x_2 \notin \cV$ and $g(x_1)=x_1$, we have
\begin{equation}
\{f,g\}_{\rm ari}^{\flat}-\{ f^{\flat},g^{\flat} \}=2(x_1 - x_3)x_3 \neq 0.
\end{equation}

\begin{remark}
A result similar to Proposition \ref{prop:V} was obtained by Racinet \cite[Corollaire A.5.4]{Racinet:these} in the non-commutative setting, however our conventions for the Ihara and ari bracket are slightly different so that we cannot use his result here. 
\end{remark}

\subsection{Proof of Theorem \ref{thm:dmls}}

We are now in a position to prove Theorem \ref{thm:dmls}.

\begin{proof}[Proof of Theorem \ref{thm:dmls}]
Note that the defining equation of the space $\cV^{\flat}$ is given by
\begin{equation}\label{eqn:v^flat}
g^{(r)}(x_1,\ldots,x_r)=(-1)^{r-1}g^{(r)}(x_r-x_{r-1},x_r-x_{r-2},\ldots,x_r-x_1,x_r).
\end{equation}
For $f\in \ls_\cQ$, since $f^\sharp (\x_1\cdots \x_i\sh \x_{i+1}\cdots \x_r)=0$ for $1\le i<r$, by Lemma \ref{lem:anti} we have $f^\sharp(\x_1\cdots \x_r)+(-1)^rf^\sharp(\x_r\cdots \x_1)=0$.
Therefore, taking $\flat$, we get 
\[(f^\sharp)^\flat(\x_1\cdots\x_r) = (-1)^{r-1}\big(anti(f^\sharp)\big)^\flat(\x_1\cdots\x_r),\]
which has the same form as \eqref{eqn:v^flat}, where $anti$ is defined in \eqref{eqn:anti}.
Thus, $\ls_\cQ\subset \cV^\flat$.
Similarly, one has $\dm_\cQ\subset \cV^\flat$.
By Proposition \ref{prop:V}, the Lie isomorphism $\flat: (\cV,\{\phantom{\cdot},\phantom{\cdot}\}_{\rm ari})\rightarrow (\cV^\flat,\{\phantom{\cdot},\phantom{\cdot}\})$ induces isomorphisms
\[
(\dm^{\sharp}_\cQ,\{ \phantom{\cdot},\phantom{\cdot}\}_{\rm ari}) \stackrel{\flat}{\longrightarrow} (\dm_\cQ,\{ \phantom{\cdot},\phantom{\cdot}\})  \ \mbox{and}\ (\ls^{\sharp}_\cQ,\{ \phantom{\cdot},\phantom{\cdot}\}_{\rm ari}) \stackrel{\flat}{\longrightarrow} (\ls_\cQ,\{ \phantom{\cdot},\phantom{\cdot}\})
\]
of Lie algebras.
Thus, Theorem \ref{thm:dmls} follows from Theorem \ref{thm:lie-ls} and Corollary \ref{cor:lie-dm}.
\end{proof}

\begin{remark}\label{rmk:another-proof}
Another proof of Theorem \ref{thm:dmls} is obtained by a version of Racinet's theorem \cite[Proposition 4.A.i]{Racinet}, together with Brown's argument \cite[\S16.1]{Brown:Anatomy} (see also \cite[Theorem 9.2]{Brown:Anatomy}). Ecalle's theory gives another approach to Racinet's theorem.
\end{remark}

\section{Comparison of Ecalle's and Brown's polar solutions}

In this section, we first define Ecalle's and Brown's polar solutions, and then, prove Theorem \ref{thm:main}.

\subsection{Ecalle's polar solutions}
Transposing $\Ad_{\rm ari}(P)$ to a map from $\ls_\cQ$ to $\dm_\cQ$ via Proposition \ref{prop:V}, we arrive at the following definition.
\begin{definition}\label{def:chi_E}
Define a morphism of Lie algebras $\chi_E: \ls_\cQ \rightarrow \dm_\cQ$ by the following commutative diagram
\[\begin{CD}
 \ls_\cQ^{\sharp} @>\Ad_{\rm ari}(P)>>   \dm_\cQ^{\sharp}   \\
@A\sharp AA       @V\flat VV\\
\ls_\cQ @>\chi_E>>  \dm_\cQ
\end{CD}.\]
\end{definition}
We can give an explicit formula for $\chi_E$ as follows.
\begin{proposition}
For $f \in \ls_\cQ$, we have
\begin{equation}\label{eqn:def-chiE}
\chi_E(f)=\sum_{n=0}^{\infty}\frac{1}{n!}\ad^n(\phi_0^\flat)(f) = f + \{f,\phi_0^\flat\}+ \frac{1}{2} \{\{f,\phi_0^\flat\},\phi_0^\flat\} + \cdots,
\end{equation}
where $\ad$ denotes the adjoint action with respect to the Ihara bracket, i.e. $\ad^n(g)(f)= \{\ad^{n-1}(g)(f),g\} $ for $n\ge1$ and $\ad^0(g)(f)=f$.
\end{proposition}
\begin{proof}
Unravelling the definition of $\chi_E$, we see that
\begin{equation}
\chi_E(f)=\left(\sum_{n=0}^{\infty}\frac{1}{n!}\ad^n(\phi_0)(f^{\sharp})\right)^{\flat}=(f^{\sharp})^{\flat}+\{f^{\sharp},\phi_0\}^{\flat}_{\rm ari}+\frac 12\{\{f^{\sharp},\phi_0\}_{\rm ari},\phi_0\}^{\flat}_{\rm ari} +\cdots .
\end{equation}
By \eqref{eq:phi0-shuffle} and Lemma \ref{lem:anti}, we have $\phi_0\in \mathcal{V}$.
Therefore, using the definition of $\Ad_{\rm ari}(P)$ together with Proposition \ref{prop:V}, we get the formula for $\chi_E(f)$ when $f\in \ls_\cQ$.
\end{proof}

Combining Theorem \ref{thm:Ecalle's-fundamental} with Proposition \ref{prop:V}, we get the next theorem.
\begin{theorem}\label{thm:chiE}
The morphism $\chi_E$ is an isomorphism of Lie algebras
\[
\chi_E: \ls_\cQ \stackrel{\cong}{\longrightarrow} \dm_\cQ.
\]
\qed
\end{theorem}
Applying $\chi_E$ to the canonical depth one element $x_1^{2k} \in \ls_\cQ$, we obtain a solution $\eta_{2k+1}:=\chi_E(x_1^{2k})$ to the double shuffle equations modulo products in weight $-2k-1$.

\subsection{Brown's polar solutions}

The following definitions are taken from \cite[\S14]{Brown:Anatomy}. For an integer $r\geq 1$, define a rational function $s_r\in \Q(x_1,\ldots,x_r)$ by
\begin{equation}
s_r=\sum_{i=0}^{r-1}(r-i) \prod_{0 \leq j \leq r, \, j \neq i}\frac{1}{(x_j-x_i)},
\end{equation}
where we set $x_0=0$.
We also define the element $\psi_0 \in \cL$ by
\[
\psi^{(r)}_0=\binom{r+1}{2}^{-1}s_r.
\]
It is shown in \cite[Proposition 14.8]{Brown:Anatomy} that the $s_r$ satisfy the identity $\{s_m,s_n\}=(m-n)s_{m+n}$. This implies that for $r\geq 3$, we have the ``Witt identity''.
\begin{equation}\label{eqn:formula-psi0}
\psi_0^{(r)}=\frac{r-1}{(r-2)(r+1)}\{\psi_0^{(1)},\psi_0^{(r-1)}\}.
\end{equation}
In particular, it follows that $\psi_0^{(r)}\in \Q^\times \ad^{r-2}(\psi_0^{(1)})(\psi_0^{(2)})$. The first few values of $\psi^{(r)}_0$ are given explicitly by
\begin{equation}\label{eqn:psi_0}
\begin{aligned}
\psi^{(1)}_0(x_1)&=\frac{1}{x_1},\quad \psi^{(2)}_0(x_1,x_2)=\frac{2 x_1-x_2}{3 x_1 (x_1-x_2) x_2},\\
\psi^{(3)}_0(x_1,x_2,x_3)&=\frac{3 x_2 x_1^2-2
   x_3 x_1^2-3 x_2^2 x_1+2 x_3^2 x_1-x_2 x_3^2+x_2^2 x_3}{6 x_1x_2x_3 (x_1-x_2)
(x_1-x_3) (x_2-x_3)}.
\end{aligned}
\end{equation}

\begin{remark}
It is announced in \cite[Theorem 14.2]{Brown:Anatomy} that the element $\psi_0$ satisfies the double shuffle equations modulo products \eqref{eqn:dm}. In particular, this means that $\psi_0$ substantially differs from Ecalle's $\phi_0^{\flat}$ (see Remark \ref{rmk:phi0}).
\end{remark}

With $\psi_0$, we can now define Brown's lift $\chi_B$.
\begin{definition} \label{dfn:constructionChiE}
For $f=f^{(d)}\in \gr^{(d)}\ls_\cQ$, define an element $\chi_B(f) \in \cL$ inductively by $\chi_B(f)^{(i)}=0$ for $i<d$, $\chi_B(f)^{(d)}:=f^{(d)}$ and 
\[
\chi_B(f)^{(d+r)}:=\frac{1}{2r}\sum_{i=1}^{r}\{ \psi_0^{(i)},\chi_B(f)^{(d+r-i)} \}
\]
for $r\geq 1$.
Extending $\chi_B$ to all of $\ls_\cQ$ by continuity and linearity, we obtain an injective linear map 
\[
\chi_B: \ls_\cQ \longrightarrow \cQ.
\]
\end{definition}

The $\chi_B(f)$ is denoted by $\widetilde{f}$ in \cite[Definition 14.3]{Brown:Anatomy}.
Clearly, the image of $\chi_B$ lies in the subspace $\cL \subset \cQ$. 
The following theorem is announced in \cite[Theorem 14.4]{Brown:Anatomy}.
\begin{theorem} \label{thm:Brown}
For $f\in\ls_\cQ$, we have $\chi_B(f) \in\dm_\cQ$, i.e. the element constructed in Definition \ref{dfn:constructionChiE} solves the double shuffle equations modulo products. \qed
\end{theorem}

Assuming the above theorem, we see that the linear map $\chi_B$ is an isomorphism\footnote{This was pointed out by Brown in a private discussion with the authors.}, since its inverse $\chi_B^{-1}$ can be described as follows (cf. \cite[Theorem 14.9]{Brown:Anatomy}). 
Given $(f^{(r)}) \neq 0\in \dm_\cQ$, let $d$ be minimal such that $f^{(d)}\neq 0$. 
Then define an element $g:=\chi_B^{-1}(f)$ recursively by
\[
g^{(d+r)}:=\begin{cases}0& r<0\\ f^{(d)} & r=0 \\ f^{(d+r)}-\chi_B(g)^{(d+r)} & r>0.\end{cases}
\]
This is well-defined, as the depth $r$ component $\chi_B(g)^{(r)}$ only involves $g^{(n)}$, for $n<d+r$. 
It is easily seen that $(g^{(r)}) \in \cL$ satisfies the linearized double shuffle equations and that $\chi_B(g)=f$. 
Since $\chi_B$ is obviously injective, this shows that $\chi_B$ is an isomorphism of $\Q$-vector spaces. 
We do not know if $\chi_B$ is an isomorphism of Lie algebras, but from the comparison with $\chi_E$ (which is a Lie algebra isomorphism!), it will follow that this is true at least in depths $\leq 3$.

\begin{remark}
The images $\chi_B(x_1^{2k})$ of the canonical depth one elements $x_1^{2k} \in \gr^{(1)}\ls_\cQ$ are denoted by $\xi_{2k+1}$ in \cite[Definition 5.1]{Brown:depth3}.\footnote{There is a typo in the definition of $\xi_{2k+1}$; $s^{(2)}$ should be $s^{(2)}/2$.} 
Interestingly, they satisfy relations; the first example is $\{\xi_3,\xi_9\}-3\{\xi_5,\xi_7\}=0$.
Moreover, for $k\geq -1$, the $\xi_{2k+1}$ coincide with the images of certain geometric derivations $\varepsilon_{2k+2}^\vee$ under a Lie algebra isomorphism $\ell' : B^1 {\rm Der}^{\varTheta} \Lie(a,b)\rightarrow (\cL,\{,\})$ up to depth four (see \cite[\S6]{Brown:depth3}). 
\end{remark}

\subsection{Comparing the two lifts}
We now prove Theorem \ref{thm:main}.

\begin{proof}[Proof of Theorem \ref{thm:main}]
Let $f\in \gr^{(d)}\ls_\cQ$.
For $r\ge1$, put $\overline{\psi}_0^{(d)}=-\frac{1}{2^d}\psi_0^{(r)}$. 
It follows from the definition of $\chi_B$ that
\begin{align*}
&\chi_B(f)^{(d)}=f, \quad \chi_B(f)^{(d+1)}=\{f,\overline{\psi}^{(1)}_0 \}, \\
&\chi_B(f)^{(d+2)}=\{f, \overline{\psi}^{(2)}_0 \}+\frac12 \{\{f,\overline{\psi}^{(1)}_0 \},\overline{\psi}^{(1)}_0\}.
\end{align*}
For $\chi_B(f)^{(d+3)}$, using $\overline{\psi}^{(3)}_0=\frac 12\{ \overline{\psi}^{(2)}_0,\overline{\psi}^{(1)}_0 \}$ (see \eqref{eqn:formula-psi0}) and the Jacobi identity, we have
\[\chi_B(f)^{(d+3)}=\{\{f, \overline{\psi}^{(2)}_0 \},\overline{\psi}^{(1)}_0\}+\frac{1}{3!} \{\{ \{f,\overline{\psi}^{(1)}_0 \}, \overline{\psi}^{(1)}_0 \},\overline{\psi}^{(1)}_0 \}.\]
On the other hand, by \eqref{eqn:def-chiE}, we have
\begin{align*}
\chi_E(f)^{(d)}&=f, \quad \chi_E(f)^{(d+1)}=\{f,\big(\phi_0^{(1)}\big)^\flat\},\\
 \chi_E(f)^{(d+2)}&=\{ f,\big(\phi_0^{(2)}\big)^\flat \}+\frac12\{\{f,\big(\phi_0^{(1)}\big)^\flat \},\big(\phi_0^{(1)}\big)^\flat\}.
\end{align*}
From \eqref{eqn:phi_0} and \eqref{eqn:psi_0}, it can be shown that for $r=1,2$, we have $(\phi_0^{(r)})^{\flat}=\overline{\psi}^{(r)}_0$, and hence, $\chi_E(f)^{(d+r)}=\chi_B(f)^{(d+r)}$ for $r=0,1,2$. It is straightforward to verify that the element $\phi_0$ satisfies the extra identity $\phi^{(3)}_0=\frac 12\{ \phi^{(2)}_0,\phi^{(1)}_0 \}_{\rm ari}$.
With this and the Jacobi identity one computes
\[ \chi_E(f)^{(d+3)}=\{\{f, \big(\phi_0^{(2)}\big)^\flat \},\big(\phi_0^{(1)}\big)^\flat \}+\frac{1}{3!} \{\{ \{f,\big(\phi_0^{(1)}\big)^\flat \}, \big(\phi_0^{(1)}\big)^\flat \}, \big(\phi_0^{(1)}\big)^\flat \}.\]
Again, by the equality $(\phi_0^{(r)})^{\flat}=\overline{\psi}^{(r)}_0$ for $r=1,2$, we get $\chi_E(f)^{(d+3)}=\chi_B(f)^{(d+3)}$, which completes the proof.
\end{proof}

\begin{remark}\label{rmk:diff}
The comparison between $\chi_B^{(d+4)}(f)$ and $\chi_E^{(d+4)}(f)$ is more complicated because for $r\geq 4$, the $\phi_0$ does not satisfy the Witt identity \eqref{eqn:formula-psi0}. 
However, a straightforward calculation gives 
\begin{equation}
\chi_E(f)^{(d+4)}-\chi_B(f)^{(d+4)}=\frac{1}{240}\{f,Q_4\} \neq 0
\end{equation}
for any $f\in \gr^{(r)}\ls_\cQ$, so that in particular $\chi_B$ and $\chi_E$ differ in general. 
Here $Q_4$ is defined by
\begin{equation}\label{eq:Q_4}
Q_4(x_1,\ldots,x_4)=\sum_{i \in \Z/5\Z}\frac{1}{(x_{i+1}-x_i)(x_{i+3}-x_i)(x_{i+3}-x_{i+2})(x_{i+4}-x_i)}
\end{equation}
with $x_0=0$, where the sum is over all cyclic permutations of the set $\{0,1,\ldots,4\}$. The element $Q_4$, which is a solution to the linearized double shuffle equations, will play a new role in the study of solutions to the double shuffle equations (see also \cite[Remark 14.10]{Brown:Anatomy}). 
\end{remark}

\subsection{A future work}

It is natural to ask if there are more isomorphisms between $\ls_\cQ$ and $\dm_\cQ$. 
We omit the details, but the discussions in \S3.4 and \S3.5 can be applied to the linearized Ihara action $\cp$: the group $(\cG,\circ)$ acts on the Lie algebra $(\cL,\{,\})$ on the right via the adjoint action $\Ad:\cG\times \cL \rightarrow \cL$, where $\circ$ is the so-called Ihara group law (see e.g. \cite[\S2.2]{Brown:depthgraded}).
The adjoint action in this case is given for $G\in \cG$ and $f\in\cL$ by the formula
\[\Ad(G)(f)= \left(\sum_{n=0}^{\infty}\ad^n(\psi_G)(f) \right),\]
where we set $\psi_G=\log(G)$. 
Note that each $\Ad(G)$ is a Lie isomorphism with respect to the Ihara bracket.
We can restrict to $\Ad(G): \ls_\cQ\rightarrow \cL$ and ask for conditions on $G$ such that its image is $\dm_\cQ$. The prototype is of course Ecalle's isomorphism $\Ad_{\rm ari}(P): \ls_\cQ^\sharp \rightarrow \dm_\cQ^\sharp$ in Theorem \ref{thm:Ecalle's-fundamental}, where $P$ is weight $0$ and satisfies the equations \eqref{eqn:dshmodlowerdepth}.

In a similar vein, one can try to write down an explicit element $B\in\cG$ of weight 0 such that 
\[ \chi_B(f)=\Ad(B)(f).\]
We make a first step in this direction. Define $\psi_B \in \cQ/\cQ^{(6)}$ by $\psi_B^{(0)}=1$ and 
\[
\psi_B^{(1)}=-\frac 12\psi_0^{(1)}, \quad \psi_B^{(2)}=-\frac 14\psi_0^{(2)}, \quad \psi_B^{(3)}=-\frac 18\psi_0^{(3)}, \quad \psi_B^{(4)}=-\frac{1}{18}\psi_0^{(4)}, \quad \psi_B^{(5)}=-\frac{11}{576}\psi_0^{(5)}.
\]
Note that $\psi_B^{(r)}=\big(\phi^{(r)}_0\big)^\flat$ for $r\leq 3$, but not for $r\geq 4$.
\begin{proposition}\label{prop:observation}
For $f \in \gr^{(d)}\ls_\cQ$, we have
\[\chi_B(f)\equiv \left(\sum_{n=0}^{\infty}\ad^n(\psi_B)(f) \right) \mod\cQ^{(d+5)}.\]
Furthermore, the element $B:=\exp(\psi_B)$ satisfies the equation
\begin{equation}\label{eqn:linstuffle}
\begin{gathered}
B(\x_1\cdots\x_i \sh \x_{i+1}\cdots \x_r)=B(\x_1\cdots\x_i)B(\x_{i+1}\cdots\x_r),\\
B^\sharp(\x_1\cdots\x_i \sh \x_{i+1}\cdots \x_r)=B^\sharp(\x_1\cdots\x_i)B^\sharp(\x_{i+1}\cdots\x_r)
\end{gathered}
\end{equation}
for $0\le i\le r\le 5$, where the $\exp$ is defined by replacing $\cp_{\rm ari}$ with $\cp$ in Definition \ref{dfn:expari}.
\end{proposition}

Since Brown's lift $\chi_B: \ls_\cQ\rightarrow \dm_\cQ$ is expected to be a Lie isomorphism, Proposition \ref{prop:observation} may provide necessary conditions on $G$ (namely \eqref{eqn:linstuffle}) such that the adjoint action $\Ad(G)$ induces a Lie isomorphism $\Ad(G): \ls_\cQ\rightarrow \dm_\cQ$. However, \eqref{eqn:linstuffle} will not be sufficient: the unit element $\mathbf{1}=(1,0,0,\ldots) \in \cG$ satisfies \eqref{eqn:linstuffle} but $\Ad(\mathbf{1})$ is the identity.

\section{Anatomical decomposition of $\sigma_{2k+1}$}

In \cite{Brown:Anatomy}, Brown gave anatomical decompositions of $\sigma_3,\sigma_5,\sigma_7,\sigma_9$ with elements $\psi_{2k+1}$ (see \cite[\S 11.4]{Brown:Anatomy}) and with lifting elements $\xi_{2k+1}=\chi_B(x_1^{2k})$ (see \cite[\S14.5]{Brown:Anatomy}). 
In this section, we give anatomical decompositions of Ecalle's polar solutions which are not in Brown's list.

Let $\eta_{2k+1} = \chi_E(x_1^{2k})$ and write $\{f_1,f_2,f_3\}=\{f_1,\{f_2,f_3\}\}$ (in general, we define $\{f_1,\ldots,f_n\}=\{f_1,\{f_2,\ldots,f_n\}\}$ inductively).
Then the results are as follows.
\begin{align*}
\sigma_3 &\equiv \eta_3 \mod \cQ^{(3)},\\
\sigma_5 &\equiv \eta_5 -\frac{5}{24} \{\eta_3,\eta_3,\eta_{-1}\} \mod \cQ^{(5)},\\
\sigma_7 &\equiv \eta_7 - \frac{7}{96} \{\eta_3,\eta_5,\eta_{-1}\} - \frac{7}{48} \{\eta_5,\eta_3,\eta_{-1}\} \\
&+\frac{37}{86400}\{ \eta_{-1},\eta_{-1},\eta_{-1},\eta_3,\eta_7\}+ \frac{3}{3200} \{\eta_{-1},\eta_{-1},\eta_3,\eta_7,\eta_{-1}\}\\
& + \frac{1}{1920} \{\eta_{-1},\eta_3,\eta_{-1},\eta_{7},\eta_{-1}\} -\frac{1}{2304}\{\eta_{-1},\eta_{-1},\eta_5 ,\eta_5,\eta_{-1}\} \\
& +\frac{5}{6912}\{\eta_5,\eta_{-1},\eta_{-1},\eta_5,\eta_{-1}\}-  \frac{661}{14400}\{\eta_3,\eta_{-1},\eta_3,\eta_3,\eta_{-1}\}\\
&+\frac{661}{28800} \{\eta_{-1},\eta_3,\eta_3,\eta_3,\eta_{-1}\}  \mod \cQ^{(7)},\\
\sigma_9&\equiv \eta_9 -\frac{5}{36}\{\eta_7,\eta_3,\eta_{-1}\} - \frac{7}{144} \{\eta_5,\eta_5,\eta_{-1}\} - \frac{5}{108}\{\eta_3,\eta_7,\eta_{-1}\} \mod \cQ^{(5)}.
\end{align*}


\begin{thebibliography}{99}

\bibitem{AlT} A.~Alekseev, C.~Torossian \textit{The Kashiwara-Vergne conjecture and Drinfeld's associators}, Ann. of Math. (2) 175 (2012), no. 2, 415--463.


\bibitem{Brown:MTM} F.~Brown, \textit{Mixed Tate motives over $\Z$},
Ann. of Math. (2) 175 (2012), no. 2, 949--976.

\bibitem{Brown:depthgraded} F.~Brown, \textit{Depth-graded motivic multiple zeta values}, arXiv:1301.3053.

\bibitem{Brown:depth3} F.~Brown, \textit{Zeta elements in depth 3 and the fundamental Lie algebra of the infinitesimal Tate curve}, Forum Math. Sigma 5 (2017), e1, 56 pp.

\bibitem{Brown:Anatomy} F.~Brown, \textit{Anatomy of an associator}, arXiv:1709.02765v1.

\bibitem{DeligneGoncharov:MTM} P.~Deligne, A.~Goncharov, 
\textit{Groupes fondamentaux motiviques de Tate mixte,} Ann. Sci. \'Ecole Norm. Sup. (4) 38 (2005), no. 1, 1--56.

\bibitem{Drinfeld} V.~G.~Drinfel'd, \textit{On quasitriangular quasi-Hopf algebras and on a group that is closely connected with ${\rm Gal}(\overline{\mathbb{Q}}/\mathbb{Q})$}, Algebra i Analiz 1 (1989) 114--148 (in Russian), English translation in Leningrad Math. J. 1 (1990) 1419--1457.

\bibitem{Ecalle:ARI} J.~Ecalle, \textit{ARI/GARI, la dimorphie et l'arithm\'etique des multiz\^etas: un premier bilan}, J. Th\'eor. Nombres Bordeaux, 15 (2003), no. 2, 411--478.


\bibitem{Furusho:dsh} H.~Furusho, \textit{Double shuffle relation for associators}, Ann. of Math. (2) 174 (2011), no. 1, 341--360.

\bibitem{Furusho:around} H.~Furusho \textit{Around associators,} Automorphic forms and Galois representations. Vol. 2, 105--117, London Math. Soc. Lecture Note Ser., 415, Cambridge Univ. Press, Cambridge, 2014.

\bibitem{Goncharov1} A.~Goncharov, \textit{Multiple polylogarithms and mixed Tate motives}, arXiv:math/0103059.

\bibitem{Goncharov2} A.~Goncharov, \textit{Periods and mixed motives}, arXiv:math/0202154.
%

\bibitem{Ihara:Tatetwists} Y.~Ihara, \textit{The Galois representation arising from $\mathbb{P}^1 \setminus \{0,1,\infty\}$ and Tate twists of even degree}, Galois groups over $\Q$ (Berkeley, CA, 1987), 299--313, Math. Sci. Res. Inst. Publ., 16, Springer, New York, 1989. 

\bibitem{Ihara:RIMS} K.~Ihara, \textit{Derivation and double shuffle relations for multiple zeta values}, RIMS preprint 1549 (2007), 47--63.

\bibitem{IKZ} K.~Ihara, M.~Kaneko, D.~Zagier, \textit{Derivation and double shuffle relations for multiple zeta values}, Compos. Math. 142 (2006), no. 2, 307--338.


\bibitem{OudomGuin} J.-M.~Oudom, D. Guin, \textit{On the Lie enveloping algebra of a pre-Lie algebra}, J. K-Theory 2 (2008), no.1, 147--167.

\bibitem{Racinet:these} G.~Racinet, \textit{S\'eries g\'en\'eratrices non-commutatives de polyz\^etas et associateurs de Drinfeld}. Universit\'e de Picardie Jules Verne, 2000.

\bibitem{Racinet} G.~Racinet, \textit{Doubles m\'elanges des polylogarithmes multiples aux racines de l'unit\'e}, Publ. Math. Inst. Hautes \'Etudes Sci. No. 95 (2002), 185--231. 

\bibitem{Reutenauer} C.~Reutenauer, \textit{Free Lie algebras}, London Mathematical Society Monographs. New Series, 7. Oxford Science Publications. The Clarendon Press, Oxford University Press, New York, 1993. xviii+269 pp.

\bibitem{SalernoSchneps} A.~Salerno, L.~Schneps, \textit{Mould theory and the double shuffle Lie algebra structure}, Springer Proceedings in Mathematics and Statistics. Eds H. Gangl, K. Ebrahimi-Fard and J. Burgos Gil. Periods in Quantum Field Theory and Arithmetic. To appear.

\bibitem{Schneps:ARI} L.~Schneps, \textit{ARI,GARI, ZIG and ZAG: An introduction to Ecalle's theory of multiple zeta values}, arXiv:1507.01534v1.


%
\bibitem{Terasoma} T.~Terasoma, \textit{Mixed Tate motives and multiple zeta values}.
Invent. Math. 149 (2002), no. 2, 339--369.
	
\bibitem{Zagier} D.~Zagier, \textit{Values of zeta functions and their applications}, First European Congress of Mathematics, Vol. II (Paris, 1992), 497--512,
Progr. Math., 120, Birkh\"auser, Basel, 1994.	


\end{thebibliography}
\end{document}